\documentclass[a4paper,11pt,reqno,twoside]{article}
\usepackage{amssymb}
\usepackage[frenchb,english]{babel}
\usepackage{amscd}
\usepackage{amsmath,amsthm,amsfonts,amssymb,graphicx}
\usepackage[colorlinks,linkcolor=blue,anchorcolor=blue,citecolor=green]{hyperref}
\usepackage{mathrsfs}
\usepackage{fancyhdr}
\usepackage{subfigure}
\usepackage{bm}
\usepackage{multicol}
\usepackage{picins}
\usepackage{abstract}

\makeatother

\begin{document}
\theoremstyle{plain}
\newtheorem{Definition}{Definition}[section]
\newtheorem{Proposition}{Proposition}[section]
\newtheorem{Property}{Property}[section]
\newtheorem{Theorem}{Theorem}[section]
\newtheorem{Lemma}[Theorem]{\hspace{0em}\bf{Lemma}}
\newtheorem{Corollary}[Theorem]{Corollary}
\newtheorem{Remark}{Remark}[section]
\newtheorem{Example}{Example}[section]

\setlength{\oddsidemargin}{ 1cm}  
\setlength{\evensidemargin}{\oddsidemargin}
\setlength{\textwidth}{13.50cm}
\vspace{-.8cm}

\noindent  {\LARGE Rawnsley's $\varepsilon$-function on some Hartogs type domains over bounded
symmetric domains and its applications }\\\\

\noindent\text{Enchao Bi$^{1}$$^{*}$,\; Zhiming Feng$^2$,\; Guicong S{u$^3$},\;\;\&\; Zhenhan T{u$^3$} }\\

\noindent\small {${}^1$School of Mathematics and Statistics, Qingdao
University, Qingdao, Shandong 266071, P.R. China} \\

\noindent\small {${}^2$School of Mathematical and Information Sciences, Leshan Normal University, Leshan, Sichuan 614000, P.R. China } \\

\noindent\small {${}^3$School of Mathematics and Statistics, Wuhan
University, Wuhan, Hubei 430072, P.R. China} \\

\renewcommand{\thefootnote}{{}}
\footnote{\hskip -16pt {$^{*}$Corresponding author, email: bienchao@whu.edu.cn \\}}
\\



\normalsize \noindent\textbf{Abstract}\quad {
The purpose of this paper is twofold. Firstly, we will compute the explicit expression of the Rawnsley's $\varepsilon$-function $\varepsilon_{(\alpha,g(\mu;\nu))}$ of $\big(\big(\prod_{j=1}^k\Omega_j\big)^{{\mathbb{B}}^{d_0}}(\mu),g(\mu;\nu)\big)$, where $g(\mu;\nu)$ is a K\"{a}hler metric associated with the K\"{a}hler potential $-\sum_{j=1}^k\nu_j\ln N_{\Omega_j}(z_j,\overline{z_j})^{\mu_j}-\ln(\prod_{j=1}^kN_{\Omega_j}(z_j,\overline{z_j})^{\mu_j}-\|w\|^2)$ on the generalized Cartan-Hartogs domain $\big(\prod_{j=1}^k\Omega_j\big)^{{\mathbb{B}}^{d_0}}(\mu)$ and obtain necessary and sufficient conditions for $\varepsilon_{(\alpha,g(\mu;\nu))}$ to become a
polynomial in $1-\|\widetilde{w}\|^2$ (see the definition \eqref{00000} for $\widetilde{w}$). Secondly, we study the Berezin quantization on $\big(\prod_{j=1}^k\Omega_j\big)^{{\mathbb{B}}^{d_0}}(\mu)$ with the metric $ g(\mu;\nu)$.

\vskip 10pt

\textperiodcentered \; Cartan-Hartogs domains \textperiodcentered \;
K\"{a}hler metrics \textperiodcentered \; Berezin quantization

\vskip 10pt

  \textperiodcentered \, 32M15  \textperiodcentered \, 32Q15

\setlength{\oddsidemargin}{-.5cm}  
\setlength{\evensidemargin}{\oddsidemargin}
\pagenumbering{arabic}
\renewcommand{\theequation}
{\arabic{section}.\arabic{equation}}
\vskip 5pt
\setcounter{equation}{0}
\section{{Introduction}}
Recently, Berezin quantization has received a lot of attention (e.g., see Cahen-Gutt-Rawnsley \cite{CGR},  Engli\v{s} \cite{E0}, Loi-Mossa \cite{Loi-Mossa} and Zedda \cite{Zedda}). Roughly, a quantization is a construction of a quantum system from the classical mechanics of a system. In 1927, Weyl made an attempt at a quantization known as Weyl quantization. His original ideal is associating a self-adjoint operators on a separable Hilbert space with functions on a symplectic manifold and some certain commutations are fulfilled. Later on, Berezin \cite{Berezin} raised a new quantization procedure, i.e., Berezin quantization. A Berezin quantization on a K\"{a}hler manifold $(\Omega,\omega)$  is given by a family of associative algebra $\mathcal{A}_{h}$ where the parameter $h$ runs through a set E of the positive reals with $0$ in its closure and moreover there exist a subalgebra $\mathcal{A}$ of $\bigoplus\{\mathcal{A}_{h};\;h\in E\}$ such that some properties are satisfied (see Berezin \cite{Berezin} for details). More precisely, we call an associative algebra with involution $\mathcal{A}$ a quantization of $(\Omega,\omega)$ if the following properties are satisfied.

$(1)$ There exist a family of associative algebras $\mathcal{A}_{h}$ of functions on $\Omega$ where the parameter $h$ runs through a set E of the positive reals with $0$ in its closure. Moreover $\mathcal{A}$ is a subalgebra of $\bigoplus\{\mathcal{A}_{h};\;h\in E\}$.

$(2)$ For each $f\in \mathcal{A}$ which will be written $f(h,x)$ ($h\in E$, $x\in \Omega$ ) such that $f(h,\cdot)\in \mathcal{A}_{h}$, the limit $$\lim_{h\rightarrow0+}f(h,x)=\varphi(f)(x)$$ exists.

$(3)$ $\varphi(f*g)=\varphi(f)\cdot\varphi(g)$, $\varphi(h^{-1}(f*g-g*f))=\frac{1}{i}\{\varphi(f),\varphi(g)\}$ for $f,g\in\mathcal{A}$. Here $*$ and $\{,\}$ denote the product of $\mathcal{A}$ and the Poisson bracket.

$(4)$ For any two points $x_{1},x_{2}\in\Omega$, there exists $f\in \mathcal{A}$ such that $\varphi(f)(x_{1})\neq \varphi(f)(x_{2})$.

Suppose $D$ is a bounded domain in $\mathbb{C}^n$ and $\varphi$
is a strictly plurisubharmonic function on $D$. Let $g$ be a
K\"{a}hler  metric  on $D$ associated with the K\"{a}hler form
$\omega=\frac{\sqrt{-1}}{2\pi}\partial\overline{\partial}\varphi$.
For $\alpha>0$, let $\mathcal{H}_{\alpha}$ be the weighted Hilbert
space of square integrable holomorphic functions on $(D, g)$ with
the weight $\exp\{-\alpha \varphi\}$, that is,
$$\mathcal{H}_{\alpha}:=\left\{ f\in \textmd{Hol}(D): \int_{D}|f|^2\exp\{-\alpha \varphi\}\frac{\omega^n}{n!}<+\infty\right\},$$
where $\textmd{Hol}(D)$ denotes the space of holomorphic functions
on $D$. Let $K_{\alpha}(z,\overline{z})$ be the Bergman kernel
(namely, the reproducing kernel) of the Hilbert space
$\mathcal{H}_{\alpha}$ if $\mathcal{H}_{\alpha}\neq \{0\}$. The
Rawnsley's $\varepsilon$-function (see \cite{CGR}) on $D$ associated with the metric
$g$  is defined by
\begin{equation}\label{eq1.4}
 \varepsilon_{(\alpha,g)}(z):=\exp\{-\alpha \varphi(z)\}K_{\alpha}(z,\overline{z}),\;\; z\in D.
\end{equation}

Note the Rawnsley's $\varepsilon$-function depends only on the
metric $g$ and not on the choice of the K\"{a}hler potential
$\varphi$. The asymptotics of the Rawnsley's $\varepsilon$-function $
\varepsilon_{\alpha}$ was expressed in terms of the parameter
$\alpha$ for compact manifolds by Catlin \cite{Cat} and Zelditch
\cite{Zeld} (for $\alpha\in \mathbb{N}$) and for non-compact
manifolds by Ma-Marinescu \cite{MM07,MM08, MM12}. In some particular case
it was also proved by Engli\v{s} \cite{E1,E2}. Especially, when the function $ \varepsilon_{(\alpha,g)}(z)$
is a positive constant on $D$, the metric $\alpha g$ is called balanced.

In order to establish a quantization procedure on a noncompact manifold $(D,g)$, we firstly give the following two conditions (refer to \cite{E0} and \cite{Zedda}):\\
\\
($\uppercase\expandafter{\romannumeral1}$). The function $\exp\{-D_{g}(z,u)\}$ is globally defined on $D\times D$,  $\exp\{-D_{g}(z,u)\}\leq1$ and $\exp\{-D_{g}(z,u)\}$ $=1$ if and only if $z=u$, here $D_{g}(z,u)$ denotes the Calabi's diastasis function (see Calabi \cite{calabi}) which is defined by
$$D_{g}(z,u):=\varphi(z,\overline{z})+\varphi(u,\overline{u})-\varphi(z,\overline{u})-\varphi(u,\overline{z}).$$
($\uppercase\expandafter{\romannumeral2}$). The function $a(z,\overline{z})$ admits a sesquianalytic extension on $D\times D$, that is,
$$a(z,\overline{u})=\exp\{-\varphi(z,\overline{u})\}$$
and moreover there exists a infinite set $E$ of integers such that for all $\alpha\in E$, $z,\;u \in D$,
$$\varepsilon_{(\alpha, g)}(z,\overline{u})=\exp\{-\alpha \varphi(z)\}K_{\alpha}(z,\overline{z})=\alpha^{n}+B(z,\overline{u})\alpha^{n-1}+C(\alpha,z,\overline{u})\alpha^{n-2},$$
where $B(z,\overline{u})$ and $C(\alpha,z,\overline{u})$ are sesquianalytic functions in $z$, $\overline{u}$ which satisfy
$$\sup\limits_{z,u\in D}\vert B(z,\overline{u})\vert<+\infty,\;\;\sup\limits_{z,u\in D,\alpha\in E}
\vert C(\alpha,z,\overline{u}) \vert<+\infty.$$

If $\alpha g$ are balanced metrics on $D$ for $\alpha\in E$, by definition $(D,g)$ automatically satisfies the condition $(\uppercase\expandafter{\romannumeral2}).$ When the condition $(\uppercase\expandafter{\romannumeral1})$ is satisfied and $\alpha g$ are balanced metrics for $\alpha\in E$ on $D$, Berezin \cite{Berezin} was able to establish a quantization procedure on $(\Omega,g)$. In 1996, Engli\v{s} \cite{E0} extended the Berezin quantization
to the case when the above conditions $(\uppercase\expandafter{\romannumeral1})$ and $(\uppercase\expandafter{\romannumeral2})$  are satisfied.
In 2012, Loi-Mossa \cite{Loi-Mossa} proved that the above conditions $(\uppercase\expandafter{\romannumeral1})$ and $(\uppercase\expandafter{\romannumeral2})$
 are satisfied by any homogeneous bounded domain $\Omega$ equipped with a homogeneous K\"{a}hler  metric $g$ and thus the homogeneous bounded domain $(D,g)$ must admit a Berezin quantization also.

 Later on, Loi-Mossa (see \cite{Loi-Mossa 2015} Theorem 1.2) also gave necessary and sufficient condition for a homogeneous K\"{a}hler manifold to admit a Berezin quantization. They also prove that a contractible homogeneous K\"{a}hler manifold (i.e., all the products $(
\Omega,g)\times(\mathbb{C}^m; g_0)$ where $(\Omega
, g)$ is an homogeneous bounded domain and
$g_0$ is the standard
flat metric) admits a Berezin quantization. However, except for the above cases, the known instances when the above conditions $(\uppercase\expandafter{\romannumeral1})$ and $(\uppercase\expandafter{\romannumeral2})$
are satisfied are very few (see Engli\v{s} \cite{E0} and Loi-Mossa \cite{Loi-Mossa}). So it is interesting to find more complete noncompact K\"{a}hler manifolds which a quantization can be carried out.
\vskip 5pt

Let $\Omega_i\subseteq \mathbb{C}^{d_i}$ be an irreducible bounded
symmetric domain $(1\leq i\leq k)$. For given positive integer
$d_0$,  positive real numbers $\mu_i$ $(1\leq i\leq k)$, the
generalized Cartan-Hartogs domain $ \big(\prod_{j=1}^k\Omega_j
\big)^{{\mathbb{B}}^{d_0}}(\mu)$ is defined by
\begin{equation}\label{eq1.1}
 \big(\prod_{j=1}^k\Omega_j\big)^{\mathbb{B}^{d_0}}(\mu):=\left\{(z,w)\in \prod_{j=1}^k \Omega_j\times \mathbb{B}^{d_0}
  : \|w\|^2<\prod_{j=1}^kN_{\Omega_j}(z_j,\overline{z_j})^{\mu_j} \right\},
\end{equation}
where $\mu=(\mu_1,\ldots,\mu_k) \in \mathbb{(R_+)}^k,\;
z=(z_1,\ldots,z_k)\in
\mathbb{C}^{d_1}\times\cdots\times\mathbb{C}^{d_k}$, $\|\cdot\|$ is
the standard Hermitian norm in $\mathbb{C}^{d_0}$,
$N_{\Omega_j}(z_j,\overline{z_j})$ is the generic norm of $\Omega_j$
$(1\leq i\leq k)$ and $\mathbb{B}^{d_0}:=\{w\in
\mathbb{C}^{d_0}:\|w\|^2<1\}$. For the reference of the generalized
Cartan-Hartogs domains, see Feng-Tu \cite{FT2}, Tu-Wang \cite{TW}
and Wang-Hao \cite{WH}.

\vskip 5pt

Let $\nu=(\nu_1,\nu_2,\ldots,\nu_k)$ with $\nu_j>-1$ for $1\leq j\leq k$. For the generalized Cartan-Hartogs domain $ \big(\prod_{j=1}^k\Omega_j
\big)^{{\mathbb{B}}^{d_0}}(\mu)$, define
\begin{equation}\label{eq1.2}
 \Phi(z,w):=-\sum_{j=1}^k\nu_j\ln N_{\Omega_j}(z_j,\overline{z_j})^{\mu_j} -\ln\left(\prod_{j=1}^kN_{\Omega_j}(z_j,\overline{z_j})^{\mu_j}-\|w\|^2\right).
\end{equation}
The K\"{a}hler form $\omega(\mu;\nu)$ on
$\big(\prod_{j=1}^k\Omega_j\big)^{{\mathbb{B}}^{d_0}}(\mu)$ is
defined by
\begin{equation}\label{eq1.3}
 \omega(\mu;\nu):=\frac{\sqrt{-1}}{2\pi}\partial
\overline{\partial}\Phi.
\end{equation}
 The metric  $g(\mu;\nu)$ on  $\big(\prod_{j=1}^k\Omega_j\big)^{{\mathbb{B}}^{d_0}}(\mu)$
associated with  $\omega(\mu;\nu)$ is given by
$$ds^2=\sum_{i,j=1}^{n}\frac{\partial^2\Phi}{\partial Z_i \partial
\overline{Z_j}}dZ_i\otimes d\overline{Z_j},$$
where
$$n=\sum_{j=0}^kd_j,\;\;  Z=(Z_1,\ldots,Z_n):=(z,w).$$
If $\nu=0$, then  the metric  $g(\mu; \nu)$ becomes the standard canonical metric (e.g., see Bi-Tu \cite{Bi}, Feng-Tu \cite{FT, FT2}, Loi-Zedda \cite{LZ} and Zedda \cite{Zed, Zedda}). In this paper, we will focus our attention on the  metric  $g(\mu;\nu)$ on  $\big(\prod_{j=1}^k\Omega_j\big)^{{\mathbb{B}}^{d_0}}(\mu)$. In the following, we also present some new notations which will not be explicated in this section. Please refer to the Section $2$ for details.

\vskip 5pt
For the generalized Cartan-Hartogs domain $ \big(
\big(\prod_{j=1}^k\Omega_j \big)^{{\mathbb{B}}^{d_0}}(\mu), g(\mu;\nu)
\big)$ in the case of $\nu=0$, Feng-Tu \cite{FT2} proved that the Rawnsley's
$\varepsilon$-function admits the following expansion :
\begin{Theorem}[Feng-Tu \cite{FT2}]
Let $\Omega_i$ be an irreducible bounded symmetric domain  in
$\mathbb{C}^{d_i}$, and denote the
generic norm $N_{\Omega_i}(z_i,\overline{z_i})$, the dimension $d_i$
and the genus $p_i$ for $\Omega_i$ $(1\leq i \leq k)$. Set
$n=\sum_{j=0}^kd_j$, $d=\sum_{j=1}^kd_j$ and
$\alpha>\max\left\{n,\frac{p_1-1}{\mu_1(1+\nu_1)},\ldots,\frac{p_k-1}{\mu_k(1+\nu_k)}\right\}$.
 Then the Rawnsley's $\varepsilon$-function associated with
$\left(\big(\prod_{j=1}^k\Omega_j\big)^{{\mathbb{B}}^{d_0}}(\mu),g(\mu;0)\right)$
can be written as
\begin{equation}\label{eq1.5}
  \varepsilon_{(\alpha,g(\mu;0))}(z,w)=\frac{1}{\prod_{i=1}^{k}\mu_{i}^{d_{i}}}\sum_{j=0}^{d}\frac{D^{j}\widetilde{\chi}(d)}{j!}(1-
  {\Vert\widetilde{w}\Vert}^2)^{d-j}(\alpha-n)_{j+d_{0}},
\end{equation}
where
\begin{equation}\label{00000}
\widetilde{w}:=\frac{w}{\prod_{j=1}^kN_{\Omega_j}(z_j,\overline{z_j})^{\frac{\mu_j}{2}}}.
\end{equation}
\end{Theorem}

So, firstly, we will compute the expression of the Rawnsley's
$\varepsilon$-function for any $\nu>-1$ as follows.

\begin{Theorem}\label{Th:2.4}{Let $\Omega_i$ be an irreducible bounded symmetric domain  in
$\mathbb{C}^{d_i}$ in its Harish-Chandra realization, and denote the
generic norm $N_{\Omega_i}(z_i,\overline{z_i})$, the dimension $d_i$
and the genus $p_i$ for $\Omega_i$ $(1\leq i \leq k)$. For $\alpha>\max\{n,\frac{p_1-1}{\mu_1(1+\nu_1)},\ldots,\frac{p_k-1}{\mu_k(1+\nu_k)}\}$.
Then the Rawnsley's $\varepsilon$-function associated with
$\big(\big(\prod_{j=1}^k\Omega_j\big)^{{\mathbb{B}}^{d_0}}(\mu),
g(\mu;\nu)\big)$ can be written as
\begin{equation}\label{eq2.37}
     \varepsilon_{(\alpha,g(\mu;\nu))}(z_1,\ldots,z_k,w)=(\alpha-n)_n(1-\|\widetilde{w}\|^2)^{\alpha}
\sum_{t=0}^{+\infty}\psi(\alpha,t)\frac{(\alpha)_t}{t!}\|\widetilde{w}\|^{2t},
\end{equation}
where
\begin{equation*}
\widetilde{w}=\frac{w}{\prod_{j=1}^kN_{\Omega_j}(z_j,\overline{z_j})^{\frac{\mu_j}{2}}},\;\psi(x,y)=\frac{\prod_{i=1}^k\chi_i(\mu_i((1+\nu_i)x+y)-p_i)}{\prod_{i=1}^k\mu_i^{d_i}\sum_{t=0}^d\sigma(t)(x-n)_{d-t}(x+y-t)_t},
\end{equation*}
and
\begin{equation}\label{f1}
   \sigma(t)=\sum_{\sum_{i=1}^kt_i=d-t\atop
t_i\geq 0,1\leq i\leq k}\prod_{i=1}^k{d_i\choose t_i}\nu_i^{t_i}.
\end{equation}

 }\end{Theorem}
\vskip 5pt

Obviously, when $\nu=0$, the Rawnsley's
$\varepsilon$-function $\varepsilon_{(\alpha,g(\mu;\nu))}$ of
$\big(\big(\prod_{j=1}^k\Omega_j\big)^{{\mathbb{B}}^{d_0}}(\mu),
g(\mu;\nu)\big)$ becomes a polynomial in $1-\|\widetilde{w}\|^2$ (see \eqref{eq1.5}). However, by Theorem \ref{Th:2.4}, we know that the Rawnsley's
$\varepsilon$-function $\varepsilon_{(\alpha,g(\mu;\nu))}$ may not be a polynomial in $1-\|\widetilde{w}\|^2$ for a general $\nu$. So we are interested in finding some $\nu_{0}$ such that the Rawnsley's
$\varepsilon$-function is a polynomial in $1-\|\widetilde{w}\|^2$.

In this paper, with the expression of the Rawnsley's
$\varepsilon$-function $\varepsilon_{(\alpha,g(\mu;\nu))}$ of
$\big(\big(\prod_{j=1}^k\Omega_j\big)^{{\mathbb{B}}^{d_0}}(\mu),
g(\mu;\nu)\big)$ (see Theorem \ref{Th:2.4}), we obtain necessary and
sufficient conditions for $\varepsilon_{(\alpha,g(\mu;\nu))}$ being a
polynomial in $1-\|\widetilde{w}\|^2$.

\begin{Corollary}\label{Th:3.1}{
Let $\Omega_i\subseteq \mathbb{C}^{d_i}$ be an irreducible bounded symmetric domain, and denote the
generic norm $N_{\Omega_i}(z_i,\overline{z_i})$, the dimension $d_i$
and the genus $p_i$ for $\Omega_i$ $(1\leq i \leq k)$. Set
$n=\sum_{j=0}^kd_j$ and $d=\sum_{j=1}^kd_j$. Then the Rawnsley's $\varepsilon$-function associated with
$\left(\big(\prod_{j=1}^k\Omega_j\big)^{{\mathbb{B}}^{d_0}}(\mu),g(\mu;\nu)\right)$
can be written as
\begin{equation}\label{eq2.38}
    \varepsilon_{(\alpha,g(\mu;\nu))}(z_1,\ldots,z_k,w)=\sum_{j=0}^dc_j(\alpha)\left(1-{\|\widetilde{w}\|^2}\right)^{d-j}
\end{equation}
if only and if
$$\alpha>\max\left\{n,\frac{p_1-1}{\mu_1(1+\nu_1)},\ldots,\frac{p_k-1}{\mu_k(1+\nu_k)}\right\}$$
and
\begin{equation}\label{eq2.39}
  \phi(x)=(x-d)_d\psi(\alpha,x-\alpha)=\frac{(x-d)_d\prod_{i=1}^k\chi_i(\mu_i(\nu_i\alpha+x)-p_i)}{\prod_{i=1}^k\mu_i^{d_i}\sum_{t=0}^d\sigma(t)(\alpha-n)_{d-t}(x-t)_t}
\end{equation}
 is a polynomial in $x$.

Furthermore, \eqref{eq2.38} can be re-written as
\begin{equation}\label{eq2.40}
    \varepsilon_{(\alpha,g(\mu;\nu))}(z_1,\ldots,z_k,w)=\sum_{j=0}^d\frac{D^j\phi(d)}{j!}(\alpha-n)_{d_0+j}\left(1-{\|\widetilde{w}\|^2}\right)^{d-j},
\end{equation}
where $D^j\phi(x)$ denotes the  $j$-order difference of $\phi$ at
$x$, that is
\begin{equation}\label{eq2.41}
D^j\phi(x)=\sum_{l=0}^j{j\choose l}(-1)^l\phi(x-l).
\end{equation}
 }\end{Corollary}

For the Cartan-Hartogs domain $(\Omega^{\mathbb{B}^{d_{0}}}(\mu),\beta g(\mu;\nu))$, Zedda \cite{Zedda} proved that $\big(\Omega^{\mathbb{B}^{d_{0}}}(\mu),\beta g(\mu;0))$ admits a Berezin quantization. So it is natural to ask whether there is another $\nu_{0}$ such that
$$\big(\big(\prod_{j=1}^k\Omega_j\big)^{{\mathbb{B}}^{d_0}}(\mu),\\  g(\mu;\nu_0)\big)\;\;\; (k\geq 1 )$$
admits a Berezin quantization. In this paper, we also prove the following results.

\begin{Theorem}\label{Th:1.7}
Let $\Omega_{i}\subseteq \mathbb{C}^{d_{i}}$ be an irreducible
bounded symmetric domain, and denote  the rank $r_{i}$, the characteristic
multiplicities $a_{i}, b_{i}$, the dimension $d_{i}$ and the genus $p_{i}$
for $\Omega_{i}$. Let $g(\mu;\nu)$ be the metric on the generalized Cartan-Hartogs domain
$\big(\prod_{j=1}^k\Omega_j\big)^{{\mathbb{B}}^{d_0}}(\mu)$. Assume that $\mu_{i}\in W(\Omega_{i})\backslash\{0\}$, $1\leq i\leq k$,  where $W(\Omega_{i})$ are the Wallach sets defined by
$$W(\Omega_{i}):=\bigg\{0,\frac{a_{i}}{2},2\frac{a_{i}}{2},\ldots,(r_{i}-1)\frac{a_{i}}{2}\bigg\}\cup\bigg((r_{i}-1)
\frac{a_{i}}{2},+\infty\bigg).$$
If for sufficiently large $\alpha$, $\phi(x)$ (see \eqref{eq2.39}) is a polynomial in $x$,
 then the generalized Cartan-Hartogs domain $\big(\big(\prod_{j=1}^k\Omega_j\big)^{{\mathbb{B}}^{d_0}}(\mu), g(\mu;\nu)\big)$ admits a Berezin quantization.
\end{Theorem}

\vskip 5pt
When $\varepsilon_{(\alpha,g(\mu;\nu))}$ ia a positive constant, we have the following result.
\begin{Theorem}\label{co1.5}
Let $\big(\big(\prod_{j=1}^k\Omega_j\big)^{{\mathbb{B}}^{d_0}}(\mu), g(\mu;\nu)\big)$ be the generalized Cartan-Hartogs domain with canonical metric $g(\mu;\nu)$. Then for $\alpha>\max\{n,\frac{p_1-1}{\mu_1(1+\nu_1)},\ldots,\frac{p_k-1}{\mu_k(1+\nu_k)}\}$, the metric $\alpha g(\mu;\nu)$ is balanced if and only if for all $x$, $y \in \mathbb{R}$,
\begin{equation}\label{Blanced}
\prod\limits_{i=1}^{k}\chi_{i}(\mu_{i}((1+\nu_{i})x+y)-p_{i})=\left(\prod\limits_{i=1}^{k}\mu_{i}^{d_{i}}\right)\sum\limits_{t=0}^{d}
\left(\sum_{\sum_{i=1}^kt_i=t\atop
t_i\geq 0,1\leq i\leq k}\prod_{i=1}^k{d_i\choose t_i}\nu_i^{t_i}\right)(x-n)_{t}(x+y-d+t)_{d-t}.
\end{equation}
Moreover, under this situation, $\big(\big(\prod_{j=1}^k\Omega_j\big)^{{\mathbb{B}}^{d_0}}(\mu), g(\mu;\nu)\big)$ must admit a Berezin quantization for $\mu_{i}\in W(\Omega_{i})\backslash\{0\}$ $(1\leq i\leq k)$.
\end{Theorem}

\begin{Remark}
When $\nu_i=0$ $(1\leq i\leq k)$, the formula \eqref{Blanced} can be re-written as
$$\prod\limits_{i=1}^{k}\chi_{i}(\mu_{i}(x+y)-p_{i})=\prod\limits_{i=1}^{k}\mu_{i}^{d_{i}}(x+y-d)_{d}.$$
Especially, if $y=0$, by \eqref{1.3} and \eqref{s}, we have
$$\prod\limits_{i=1}^{k}\prod\limits_{j=1}^{r_i}(\mu_{i}x-p_{i}+1+(j-1)\frac{a_i}{2})_{1+b_i+(r_i-j)a_{i}}=\prod\limits_{i=1}^{k}\mu_{i}^{d_{i}}\prod\limits_{j=1}^{d}(x-j).$$
This is the exactly formula $(1.7)$ of Theroem 1.4 in Feng-Tu \cite{FT2}.
\end{Remark}
\begin{Example}
Let $k=d=d_{0}=1$. Then the generalized Cartan-Hartogs domain becomes the Thullen domain
\begin{equation*}
\mathbb{B}^{\mathbb{B}}(\mu)=\left\{(z,w)\in \mathbb{C}^{2}:{\vert z\vert}^{2}+{\vert w\vert}^{\frac{2}{\mu}}<1\right\}\;\;\;(\mu>0).
\end{equation*}

Obviously, in this special case, we have $r=1$, $a=2$, $b=0$ and $p=2$ (refer to \cite{FKKLR}). Then \eqref{eq2.39} implies that
$$
  \phi(x)=\frac{(x-1)\chi(\mu(\nu\alpha+x)-2)}{\mu\sum_{t=0}^1\sigma(t)(\alpha-2)_{1-t}(x-t)_t}.$$
By the definition of $\sigma(t)$ (see \eqref{f1}), we know that $\sigma(0)=\nu$ and $\sigma(1)=1$. Therefore, the denominator of $\phi(x)$ becomes $\mu[\nu(\alpha-2)+(x-1)]$. Moreover, by \eqref{1.3}, it follows
$$\chi(\mu(\nu\alpha+x)-2)=\mu(\nu\alpha+x)-1.$$

It is easy to see that the denominator of $\phi(x)$ equals to $\chi(\mu(\nu\alpha+x)-2)$ when $\nu=\frac{1-\mu}{2\mu}$. This means $\phi(x)=x-1$. Then by the Theorem \ref{Th:1.7}, we conclude that $(\mathbb{B}^{\mathbb{B}}(\mu), g(\mu;\frac{1-\mu}{2\mu}))$ admits a Berezin quantization.
\end{Example}

\vskip 5pt

\begin{Example}
For $\nu_{i}=0$ $(1\leq i\leq k)$, it is not hard to see that
 $$\sigma(0)=\sigma(1)=\ldots=\sigma(d-1)=0,\;\sigma(d)=1$$
 by \eqref{f1}. The formula \eqref{eq2.39} yields that
 $$\phi(x)=\frac{(x-d)_d\prod_{i=1}^k\chi_i(\mu_ix-p_i)}{\prod_{i=1}^k\mu_i^{d_i}\sigma(d)(x-d)_d}={\prod\limits_{i=1}^{k}\mu_{i}^{-d_{i}}\chi_{i}(\mu_{i}x-p_{i})}.$$
By \eqref{1.3}, we know that $\phi(x)$ is a polynomial of $x$. Therefore, $\big(\big(\prod\limits_{i=1}^{k}\Omega_{i}\big)^{\mathbb{B}^{d_{0}}}(\mu), g(\mu;0)\big)$ admits a Berezin quantization for
  $\mu_{i}\in W(\Omega_{i})\backslash\{0\}$, $1\leq i\leq k$.  In particular, when $k=1$, this happens to be the result of Zedda \cite{Zedda}.
\end{Example}
\vskip 5pt

\begin{Example}
Let $\mu=(\mu_1,\mu_2)$, $\nu_2=\frac{1-\mu_2(d_1+1)}{(d_0+d_1+1)\mu_2}$. For
$\left(\left({\Omega_1}\times \mathbb{B}\right)^{{\mathbb{B}}^{d_0}}\big(\mu),\;g(\mu;0,\nu_2)\right)$, we have
$$d_{2}=1,\;r_{2}=1,\;a_{2}=2,\;b_{2}=0,\;p_{2}=2.$$
Combined with \eqref{1.3}, we get $\chi_2(\mu_2(\nu_2\alpha+x)-p_2)=\mu_2(\nu_2\alpha+x)-1$. Since $\nu_{1}=0$, hence we can obtain that
$$\sigma(0)=\sigma(1)=\ldots=\sigma(d-2)=0,\;\sigma(d-1)=\nu_2,\; \sigma(d)=1$$
by \eqref{f1}. Consequently, the denominator of $\phi(x)$ can be expressed by
$$\mu_{1}^{d_1}\mu_{2}[(x-d)_{d}+\nu_{2}(\alpha-n)(x-d+1)_{d-1}]=\mu_{1}^{d_1}\mu_{2}(x-d+1)_{d-1}[(x-d)+\nu_{2}(\alpha-n)].$$
Therefore, $\phi(x)$ can be rewritten as
$$\phi(x)=\frac{(x-d)_d\chi_1(\mu_1x-p_1)[\mu_2(\nu_2\alpha+x)-1]}{\mu_{1}^{d_1}\mu_{2}(x-d+1)_{d-1}[(x-d)+\nu_{2}(\alpha-n)]}.$$
Moreover, it is not hard to see $\mu_2(\nu_2\alpha+x)-1=\mu_{2}[(x-d)+\nu_{2}(\alpha-n)$ for $\nu_2=\frac{1-\mu_2(d_1+1)}{(d_0+d_1+1)\mu_2}$. Hence we can see
$$\phi(x)=\mu_1^{-d_1}(x-d)\chi_1(\mu_1 x-p_1)$$
by \eqref{1.3}, which means that $\phi(x)$ is a polynomial in $x$. So $\left(\left({\Omega_1}\times \mathbb{B}\right)^{{\mathbb{B}}^{d_0}}\big(\mu),\; g(\mu;0,\nu_2)\right)$ admits a Berezin quantization for $\mu_{1}\in W(\Omega_{1})\backslash\{0\}$, $\mu_2>0$.
\end{Example}

\vskip 5pt

The paper is organized as follows. In Section 2, we will calculate the explicit expression of the Rawnsley's $\varepsilon$-function. Meanwhile, using the expression of the Rawnsley's $\varepsilon$-function expansion, we obtain the necessary and sufficient conditions for $\varepsilon_{(\alpha,g(\mu;\nu))}$ to become a
polynomial in $\|\widetilde{w}\|^2$. Lastly, with the expression \eqref{eq2.40} of the Rawnsley's $\varepsilon$-function expansion of
$\big(\big(\prod_{j=1}^k\Omega_j\big)^{{\mathbb{B}}^{d_0}}(\mu),
g(\mu;\nu)\big)$, we will prove that $\big(\big(\prod_{j=1}^k\Omega_j\big)^{{\mathbb{B}}^{d_0}}(\mu), g(\mu;\nu)\big)$ admits a Berezin quantization under some conditions.\\

\setcounter{equation}{0}
\section{The Rawnsley's $\varepsilon$-function for $\big(\prod_{j=1}^k\Omega_j\big)^{{\mathbb{B}}^{d_0}}(\mu)$  with the metric $g(\mu;\nu)$}

Firstly, Let us briefly recall some basic facts on irreducible bounded symmetric domains.

Let $\Omega\subseteq \mathbb{C}^d$ be an irreducible bounded symmetric domain and let $r$, $a$, $b$, $d$, $p$, $N_{\Omega}(z,\overline{w})$ be the rank, the characteristic
multiplicities, the dimension, the genus and the
generic norm of $\Omega$. Hence we have
\begin{equation}\label{1.1}
    d=\frac{r(r-1)}{2}  a+rb+r,\quad   p=(r-1)a+b+2.
\end{equation}
For any $s>-1$, the value of the Hua integral
$\int_{\Omega}N_{\Omega}(z,\overline{z})^s dm(z)$ is given by
\begin{equation}\label{1.2}
\int_{\Omega}N_{\Omega}(z,\overline{z})^s
dm(z)=\frac{\pi^d}{C_{\Omega}\chi(s)},
\end{equation}
where $C_{\Omega}=\det(-\frac{\partial^2 N_{\Omega}}{\partial
z^t\partial\bar{z}})(0)$, $dm(z)$ denotes the Euclidean measure on
$\mathbb{C}^d$, $\chi$ is the Hua polynomial
\begin{equation}\label{1.3}
   \chi(s):=\prod_{j=1}^r\left(s+1+(j-1)\frac{a}{2}\right
   )_{1+b+(r-j)a},
\end{equation}
in which, for a  non-negative integer $m$, $(s)_m$ denotes the
raising factorial
\begin{equation}\label{s}
{(s)_m:=\frac{\Gamma(s+m)}{\Gamma(s)}=s(s+1)\ldots (s+m-1)}.
\end{equation}

Let $\mathcal{G}$ stand for the identity connected component of the
group of biholomorphic self-maps of $\Omega$, and $\mathcal{K}$ {for
the stabilizer} of the origin in $\mathcal{G}$. Under the action
$f\mapsto f\circ k \; (k\in \mathcal{K})$ of $\mathcal{K}$, the
space $\mathcal{P}$ of holomorphic polynomials on $\mathbb{C}^d$
admits the Peter-Weyl decomposition
$$\mathcal{P}=\bigoplus_{\lambda}\mathcal{P}_{\lambda},$$
 {where the summation is taken over all partitions}
$\lambda$, i.e., $r$-tuples $(\lambda_1, \lambda_2, \ldots,
\lambda_r)$ of nonnegative integers such that $\lambda_1\geq
\lambda_2 \geq \ldots \geq \lambda_r \geq 0$,  {and the spaces}
$\mathcal{P}_{\lambda}$ are $\mathcal{K}$-invariant and irreducible.
For each $\lambda$, $\mathcal{P}_{\lambda} \subset
\mathcal{P}_{|\lambda|}$, where $|\lambda|$ denotes the weight of
partition $\lambda$, i.e., $|\lambda|:=\sum_{j=1}^r \lambda_j$, and
$\mathcal{P}_{|\lambda|}$ is the space of homogeneous holomorphic
polynomials of degree $|\lambda|$.

Let
\begin{equation}\label{1.4}
    {\langle}f,g {\rangle}_{\mathcal{F}}:=\int_{\mathbb{C}^d}f(z)\overline{g(z)} d\rho_{\mathcal{F}}(z)
\end{equation}
be the Fock-Fischer inner product on the space $\mathcal{P}$ of
holomorphic polynomials on $\mathbb{C}^d$, where
\begin{equation}\label{1.5}
    d\rho_{\mathcal{F}}(z):=\exp\{-m(z,\overline{z})\}
    \frac{\left(\frac{\sqrt{-1}}{2\pi}\partial\overline{\partial}m(z,\overline{z})\right)^d}{d!}
\end{equation}
and $m(z,\overline{z}):=-\left.\frac{\partial \ln
N_{\Omega}(tz,\overline{z})}{\partial
t}\right|_{t=0}=-\left.\frac{\partial
N_{\Omega}(tz,\overline{z})}{\partial t}\right|_{t=0}.$

 For every partition $\lambda$, let $K_{\lambda}(z_1,\overline{z_2})$ be
the reproducing kernel of $\mathcal{P}_{\lambda}$ with respect to
\eqref{1.4}. The weighted Bergman kernel of the weighted Hilbert
space $A^2(\mathbb{C}^d,\rho_{\mathcal{F}})$ of square-integrable
holomorphic functions on $\mathbb{C}^d$ with the measure
$d\rho_{\mathcal{F}}$ is
\begin{equation}\label{1.6}
   \exp\{-m(z_1,\overline{z_2})\}=\sum_{\lambda}K_{\lambda}(z_1,\overline{z_2}).
\end{equation}

The kernels $K_{\lambda}(z_1,\overline{z_2})$ are related to the
generic norm  $N_{\Omega}(z_1,\overline{z_2})$ by the
Hua-Faraut-{Kor\'{a}nyi} formula
\begin{equation}\label{1.7}
N_{\Omega}(z_1,\overline{z_2})^{-s}=\sum_{\lambda}(s)_{\lambda}K_{\lambda}(z_1,\overline{z_2}),
\end{equation}
where the series converges  {uniformly} on compact subsets of
$\Omega\times\Omega$, $s\in \mathbb{C}$,  in which  $(s)_{\lambda}$
denote the generalized Pochhammer symbol
\begin{equation}\label{1.8}
   (s)_{\lambda}:=\prod_{j=1}^r\big(s-\frac{j-1}{2}a\big)_{\lambda_j}.
\end{equation}
For the proofs of above facts and additional details, we refer,
e.g.,  to \cite{FK}, \cite{FKKLR} and \cite{YLR}.

\begin{Lemma}\label{Le:2.1}{
Let $\Omega_i\subseteq\mathbb{C}^{d_i}$ be an irreducible bounded symmetric domain, and denote the
generic norm $N_{\Omega_i}$ and the genus $p_i$ for $\Omega_i$
$(1\leq i \leq k)$. For $z_i^0\in \Omega_i$, let $\phi_i$ be an
automorphism of $\Omega_i$ such that $\phi_i(z^0_i)=0$, $1\leq i\leq
k$. By \cite{WH}, the function
\begin{equation}\label{eq2.1}
   \psi(z_1,\ldots,z_k):=\prod_{i=1}^k\frac{N_{\Omega_i}(z_i^0,\overline{z_i^0})^{\frac{\mu_i}{2}}}{N_{\Omega_i}(z_i,\overline{z_i^0})^{\mu_i}}
\end{equation}
satisfies
\begin{equation}\label{eq2.2}
    |\psi(z_1,\ldots,z_k)|^2=\prod_{i=1}^k\Big(\frac{N_{\Omega_i}(\phi_i(z_i),\overline{\phi_i(z_i)})}{N_{\Omega_i}(z_i,\overline{z_i})}\Big)^{\mu_i}.
\end{equation}
Define the mapping
\begin{equation}\label{eq2.3}
    \begin{array}{rcl}
   F: \big(\prod\limits_{j=1}^k\Omega_j\big)^{{\mathbb{B}}^{d_0}}(\mu) & \longrightarrow   & \big(\prod\limits_{j=1}^k\Omega_j\big)^{{\mathbb{B}}^{d_0}}(\mu), \\
     (z_1,\ldots,z_k,w)            & \longmapsto   & (\phi_1(z_1),\ldots,\phi_k(z_k),\psi(z_1,\ldots,z_k)w).
  \end{array}
\end{equation}
Then $F$ is an isometric automorphism of
$\big(\big(\prod\limits_{j=1}^k\Omega_j\big)^{{\mathbb{B}}^{d_0}}(\mu),g(\mu;\nu)\big)$,
namely
\begin{equation}\label{eq2.4}
   \partial\overline{\partial}(\Phi(F(z_1,\ldots,z_k,w)))=\partial\overline{\partial}(\Phi(z_1,\ldots,z_k,w)),
\end{equation}
where $ \Phi(z,w):=-\sum_{j=1}^k\nu_j\ln
N_{\Omega_j}(z_i,\overline{z_i})^{\mu_j}
-\ln\left(\prod_{j=1}^kN_{\Omega_j}(z_j,\overline{z_j})^{\mu_j}-\|w\|^2\right)$
(see \eqref{eq1.2}).
 }\end{Lemma}

\begin{proof}[Proof]
It is easy to see that $F$ is an automorphism of
$\big(\prod\limits_{j=1}^k\Omega_j\big)^{{\mathbb{B}}^{d_0}}(\mu)$,
and
\begin{equation}\label{eq2.5}
    N_{\Omega_i}(\phi_i(z_i),\overline{\phi_i(z_i)})^{p_i}=J\phi_i(z_i)N_{\Omega_i}(z_i,\overline{z_i})^{p_i}\overline{J\phi_i(z_i)},
\end{equation}
where $J\phi_i(z_i)$ is the holomorphic Jacobian of the automorphism
$\phi_i$ of $\Omega_i$, $1\leq i\leq k$.

By \eqref{eq2.2} and \eqref{eq2.5}, we have
\begin{eqnarray*}
\Phi(F)
   &=& -\sum_{i=1}^k\nu_i\ln N_{\Omega_i}(\phi_i(z_i),\overline{\phi_i(z_i)})^{\mu_i}-\ln\left(\prod_{i=1}^kN_{\Omega_i}(\phi_i(z_i),\overline{\phi_i(z_i)})^{\mu_i}-\|\psi(z_1,\ldots,z_k)w\|^2\right)\\
   &=& -\sum_{i=1}^k\mu_i(1+\nu_i)\ln N_{\Omega_i}(\phi_i(z_i),\overline{\phi_i(z_i)})-\ln\left(1-\frac{\|w\|^2}{\prod_{i=1}^kN_{\Omega_i}(z_i,\overline{z_i})^{\mu_i}}\right)\\
\end{eqnarray*}
Note the formula \eqref{eq2.5} implies that
\begin{eqnarray*}
   & & -\sum_{i=1}^k\mu_i(1+\nu_i)\ln N_{\Omega_i}(\phi_i(z_i),\overline{\phi_i(z_i)})-\ln\left(1-\frac{\|w\|^2}{\prod_{i=1}^kN_{\Omega_i}(z_i,\overline{z_i})^{\mu_i}}\right)\\
   &=&-\sum_{i=1}^k\mu_i(1+\nu_i)\ln|J\phi_i(z_i)|^{\frac{2}{p_i}}-\sum_{j=1}^k\nu_j\ln N_{\Omega_j}(z_j,\overline{z_j})^{\mu_j}
-\ln\left(\prod_{j=1}^kN_{\Omega_j}(z_j,\overline{z_j})^{\mu_j}-\|w\|^2\right).\\
\end{eqnarray*}
Therefore, by the definition \eqref{eq1.2}, we obtain
$$\Phi(F)=-\sum_{i=1}^k\mu_i(1+\nu_i)\ln|J\phi_i(z_i)|^{\frac{2}{p_i}}+\Phi.$$
which implies \eqref{eq2.4} as $J\phi_i(z_i)$ $(1\leq i\leq k)$ is a holomorphic. We complete the proof.
\end{proof}
\begin{Lemma}\label{Le:2.2}{
Let $\Omega_i\subseteq\mathbb{C}^{d_i}$ be an irreducible bounded symmetric domain, and denote the
generic norm $N_{\Omega_i}(z_i,\overline{z_i})$, the dimension $d_i$
and the genus $p_i$ for $\Omega_i$ $(1\leq i \leq k)$. Then we have
\begin{equation}\label{eq2.6}
   (\partial\overline{\partial}\Phi)^n=\frac{\prod_{i=1}^k(\mu_i^{d_i}C_{d_i})}{(1-\|\widetilde{w}\|^2)^{d_0+1}}\prod_{i=1}^k\frac{\left(\nu_i+\frac{1}{1-\|\widetilde{w}\|^2}\right)^{d_i}}{N_{\Omega_i}(z_i,\overline{z_i})^{p_i+\mu_i d_0}}
   \left(\sum_{j=1}^n dZ_j\wedge d\overline{Z_j}\right)^n,
\end{equation}
where
 $$ \Phi(z,w):=-\sum_{j=1}^k\nu_j\ln N_{\Omega_j}(z_j,\overline{z_j})^{\mu_j} -\ln\left(\prod_{j=1}^kN_{\Omega_j}(z_j,\overline{z_j})^{\mu_j}-\|w\|^2\right),$$
 $$ C_{d_i}=\left.\det\left(-\frac{\partial^2\ln N_{\Omega_i}(z_i,\overline{z_i})}{\partial z_i^t\partial{\overline{z_i}}}\right)\right|_{z_i=0},\quad\widetilde{w}=\frac{w}{\prod_{j=1}^kN_{\Omega_j}(z_j,\overline{z_j})^{\frac{\mu_j}{2}}},$$
$$n=\sum_{j=0}^kd_j,\;\;  Z=(Z_1,\ldots,Z_n)=(z_1,\ldots,z_k,w).$$
 }\end{Lemma}

\begin{proof}[Proof]
It is generally known that
\begin{equation}\label{eq2.7}
    \frac{(\frac{\sqrt{-1}}{2\pi}\partial\overline{\partial}\Phi)^n}{n!}=\det\left(\frac{\partial^2\Phi}{\partial Z^t\partial
   \overline{Z}}\right)\frac{\omega_0^n}{n!},
\end{equation}
where $\omega_0=\frac{\sqrt{-1}}{2\pi}\sum_{j=1}^ndZ_j\wedge
d\overline{Z_j}$, $\frac{\partial}{\partial
Z^t}=(\frac{\partial}{\partial Z_1},\frac{\partial}{\partial
Z_2},\ldots,\frac{\partial}{\partial Z_n})^t$,
$\frac{\partial}{\partial \overline{Z}}=(\frac{\partial}{\partial
\overline{Z_1}},\frac{\partial}{\partial
\overline{Z_2}},\ldots,\frac{\partial}{\partial \overline{Z_n}})$
and $\frac{\partial^2}{\partial Z^t\partial
\overline{Z}}=\frac{\partial}{\partial Z^t}\frac{\partial}{\partial
\overline{Z}}$.

Note the formula \eqref{eq2.4} implies that
\begin{equation}\label{eq2.8}
\det\left(\frac{\partial^2\Phi(F)}{\partial Z^t\partial
   \overline{Z}}\right)=\det\left(\frac{\partial^2\Phi}{\partial Z^t\partial
   \overline{Z}}\right).
\end{equation}
By the identity
\begin{equation}\label{eq2.9}
\frac{\partial^2\Phi(F)}{\partial Z^t\partial
   \overline{Z}}=\frac{\partial F}{\partial Z^t}\frac{\partial^2\Phi}{\partial Z^t\partial
   \overline{Z}}(F(Z))\overline{\left(\frac{\partial F}{\partial Z^t}\right)}^t
\end{equation}
 and \eqref{eq2.8}, we obtain
\begin{equation}\label{eq2.10}
\det\left(\frac{\partial^2\Phi}{\partial Z^t\partial
   \overline{Z}}\right)(Z)=|JF(Z)|^2\det\left(\frac{\partial^2\Phi}{\partial Z^t\partial
   \overline{Z}}\right)(F(Z)),
\end{equation}
where
\begin{equation*}
F:=(F_1,F_2,\ldots,F_n),~~\frac{\partial F}{\partial
Z^t}:=(\frac{\partial F_1}{\partial
   Z^t},\frac{\partial F_2}{\partial
   Z^t},\ldots,\frac{\partial F_n}{\partial
   Z^t})
\end{equation*}
and
\begin{equation*}\label{eq2.11}
   JF(Z):=\det\left(\frac{\partial F}{\partial
   Z^t}\right)(Z).
\end{equation*}

Let $Z^0=(z_1^0,\ldots,z_k^0,w^0)\in
\left(\prod_{j=1}^k\Omega_j\right)^{{\mathbb{B}}^{d_0}}(\mu)$,
$\widetilde{Z^0}:=(\widetilde{z_1^0},\ldots,\widetilde{z_k^0},\widetilde{w^0})=F(Z^0)$.
By \eqref{eq2.3}, we have
$$\widetilde{Z^0}=\left(0,\ldots,0,\frac{w^0}{\prod_{i=1}^kN_{\Omega_i}(z_i^0,\overline{z_i^0})^{\frac{\mu_i}{2}}}\right)$$
and
\begin{equation}\label{eq2.12}
 |JF(Z^0)|^2=\prod_{i=1}^k|J\phi_i(z_i^0)|^2\cdot|\psi(z_1^0,\ldots,z_k^0)|^{2d_0}.
\end{equation}
Using $N_{\Omega_i}(0,z_i)=1$ and \eqref{eq2.5}, we can see that
$$|J\phi_i(z_i^0)|^2=N_{\Omega_i}(z_i^0,\overline{z_i^0})^{-p_i}$$
From \eqref{eq2.2}, \eqref{eq2.10} and \eqref{eq2.12}, we have
\begin{equation}\label{eq2.13}
    |JF(Z^0)|^2=\prod_{i=1}^k\frac{1}{N_{\Omega_i}(z_i^0,\overline{z_i^0})^{p_i+\mu_i d_0}},
\end{equation}
and
\begin{equation}\label{eq2.13}
\det\left(\frac{\partial^2\Phi}{\partial Z^t\partial
   \overline{Z}}\right)(Z^0)=\prod_{i=1}^k\frac{1}{N_{\Omega_i}(z_i^0,\overline{z_i^0})^{p_i+\mu_i d_0}}\det\left(\frac{\partial^2\Phi}{\partial Z^t\partial
   \overline{Z}}\right)(\widetilde{Z^0}).
\end{equation}

A direct calculation gives
 \begin{eqnarray}
\nonumber    & & \frac{\partial^2\Phi}{\partial Z^t\partial\overline{Z}}(0,\ldots,0,w) \\
\label{eq2.24}    &=& \left(
                                                            \begin{array}{cccc}
                                                             \mu_1 \left(\nu_1+\frac{1}{1-\|w\|^2}\right)C_{d_1}  &\cdots & 0      & 0 \\
                                                              \vdots                          &\cdots & \vdots & \vdots \\
                                                              0                               &\cdots & \mu_k\left(\nu_k+\frac{1}{1-\|w\|^2}\right)C_{d_k} & 0 \\
                                                              0  &\cdots & 0 & \frac{1}{1-\|w\|^2}I_{d_0}+\frac{1}{(1-\|w\|^2)^2}\overline{w}^{t}w \\
                                                            \end{array}
                                                          \right),
 \end{eqnarray}
where $I_{d_0}$ denotes  the $d_0\times d_0$ identity matrix,
$\overline{w}^{t}$ is  the complex conjugate transpose of the row
vector $w=(w_1,w_2,\cdots,w_{d_0})$, and
$C_{d_i}=-\left.\frac{\partial^2\ln N_{\Omega_i}}{\partial
z_i^t\partial\overline{z_i}}\right|_{z_i=0}$.

From \eqref{eq2.24}, we have
\begin{equation}\label{eq2.25}
    \det\left(\frac{\partial^2\Phi}{\partial Z^t\partial
   \overline{Z}}\right)(0,\ldots,0,w)=\frac{\prod_{i=1}^k\left(\mu_i^{d_i}\det C_{d_i}\left(\nu_i+\frac{1}{1-\|w\|^2}\right)^{d_i}\right) }{(1-\|w\|^2)^{d_0+1}}.
\end{equation}
Finally, combining \eqref{eq2.13} and \eqref{eq2.25}, we have
\eqref{eq2.6}. The proof is finished.
\end{proof}

\begin{Theorem}\label{Th:2.3}{Let $\Omega_i\subseteq \mathbb{C}^{d_i}$ be an irreducible bounded symmetric domain, and denote the
generic norm $N_{\Omega_i}$,  the genus $p_i$, the dimension $d_i$
and the Hua polynomial $\chi_i$ (see \eqref{1.3}) for $\Omega_i$
$(1\leq i \leq k)$. Endow the generalized Cartan-Hartogs domain
$\big(\prod_{j=1}^k\Omega_j\big)^{{\mathbb{B}}^{d_0}}(\mu)$ with the canonical metric $g(\mu;\nu)$. For
$\alpha>\max\left\{n,\frac{p_1-1}{\mu_1(1+\nu_1)},\ldots,\frac{p_k-1}{\mu_k(1+\nu_k)}\right\}$,
then the reproducing kernel $K_{\alpha}(Z;\overline{Z})$ of the weighted
 Hilbert space
$$\mathcal{H}_{\alpha}=\left\{ f\in \mbox{\rm Hol}\big(\big(\prod_{j=1}^k\Omega_j\big)^{{\mathbb{B}}^{d_0}}(\mu)\big):
 \int_{\left(\prod_{j=1}^k\Omega_j\right)^{{\mathbb{B}}^{d_0}}(\mu)}|f|^2\exp\{-\alpha \Phi\}
\frac{ \omega(\mu;\nu)^{n}}{n!}<+\infty\right\}$$ can be expressed as
\begin{equation}\label{eq2.26}
  K_{\alpha}(Z;\overline{Z})=\frac{(\alpha-n)_n}{{\prod_{i=1}^kN_{\Omega_i}(z_i,\overline{z_i})^{\mu_i(1+\nu_i)\alpha}}}
\sum_{t=0}^{+\infty}\psi(\alpha,t)\frac{(\alpha)_t}{t!}\|\widetilde{w}\|^{2t}.
\end{equation}
Where
\begin{equation}\label{e4.5.4}
    Z=(z_1,\ldots,z_k,w),\; d=\sum_{j=1}^kd_j,\;n=d+d_0,\quad
\widetilde{w}=\frac{w}{\prod_{j=1}^kN_{\Omega_j}(z_j,\overline{z_j})^{\frac{\mu_j}{2}}},
\end{equation}
\begin{equation}\label{e4.5.2}
    \psi(x,y)=\frac{\prod_{i=1}^k\chi_i(\mu_i((1+\nu_i)x+y)-p_i)}{\prod_{i=1}^k\mu_i^{d_i}\sum_{t=0}^d\sigma(t)(x-n)_{d-t}(x+y-t)_t},
\end{equation}
and
\begin{equation*}
   \sigma(t)=\sum_{\sum_{i=1}^kt_i=d-t\atop
t_i\geq 0,1\leq i\leq k}\prod_{i=1}^k{d_i\choose t_i}\nu_i^{t_i}.
\end{equation*}

}\end{Theorem}

\begin{proof}[Proof]
Firstly, by definition of $C_{d_i}$ and $C_{\Omega_i}$, we know that $C_{d_i}=C_{\Omega_i}$. Therefore, according to the formula \eqref{eq2.6}, we can express the inner product on
$\mathcal{H}_{\alpha}$ as follows
\begin{eqnarray}
\nonumber (f,g)  &=& \frac{\prod_{i=1}^k(\mu_i^{d_i}C_{\Omega_i})}{\pi^{n}}\int_{\left(\prod\limits_{j=1}^k\Omega_j\right)^{{\mathbb{B}}^{d_0}}(\mu)}f(Z)\overline{g(Z)} \prod_{i=1}^kN_{\Omega_i}(z_i,\overline{z_i})^{\mu_i((1+\nu_i)\alpha-d_0)-p_i}\\
\nonumber   & &
\times\Big(1-\|\widetilde{w}\|^2
\Big)^{\alpha-d_0-1}\prod_{i=1}^k\left(\nu_i+\frac{1}{1-\|\widetilde{w}\|^2}\right)^{d_i}dm(Z),
\end{eqnarray}
where $dm$ denotes the standard Euclidean measure.

Secondly, for the sake of convenience, we define $\Omega_0:={\mathbb{B}}^{d_0}$, $z_0:=w$.
Let $r_i$, $a_i, b_i$, $d_i$, $p_i$, $ \chi_i$, $(s)_\mathbf{\lambda}^{(i)}$, $N_{\Omega_i}$ and $V(\Omega_i)$ be rank, characteristic multiplicities,
the dimension, genus, Hua polynomial, generalized Pochhammer symbol, generic norm and the Euclidean volume of the irreducible bounded symmetric domain $\Omega_i$ ($0\leq
i\leq k$).

Let $\mathcal{G}_i$ be the identity connected components of
 groups of biholomorphic self-maps of $\Omega_i \subset \mathbb{C}^{d_i}$, and $\mathcal{K}_i$
be the stabilizers of the origin in $\mathcal{G}_i$, respectively. For
any $u=(u_0,\ldots,u_k)\in \mathcal{K}:= \mathcal{K}_0\times\ldots\times\mathcal{K}_k$,
we define the action
 $$ \pi(u)f(z_1,\ldots,z_k,w)\equiv f\circ u(z_1,\ldots,z_k,w):=f(u_1\circ z_1,\ldots,u_k\circ z_k,u_0\circ w)$$
 of $\mathcal{K}$, then the space $\mathcal{P}$
of holomorphic polynomials on $\prod_{j=0}^k\mathbb{C}^{d_j}$ admits
the Peter-Weyl decomposition
$$\mathcal{P}=\bigoplus_{{\ell(\lambda_i)\leq r_i\atop 0\leq i\leq k}}\mathcal{P}^{(0)}_{\lambda_0}\otimes \ldots\otimes \mathcal{P}^{(k)}_{\lambda_k},$$
where space $\mathcal{P}^{(i)}_{\lambda_i}$ is
$\mathcal{K}_i$-invariant and irreducible subspace of the space
 of holomorphic polynomials on $\mathbb{C}^{d_i}$, and $\ell(\lambda_i)$ denotes the length of partition $\lambda_i$ $(0\leq i \leq k)$.

Since $\mathcal{H}_{\alpha}$ is invariant under the action of
$\mathcal{K}$, that is, $\forall u\in
\mathcal{K}$, $(\pi(u)f,\pi(u)g)=(f,g)$,
 $\mathcal{H}_{\alpha}$ admits
an irreducible decomposition (see \cite{Far-Tho})
$$\mathcal{H}_{\alpha}=\widehat{\bigoplus_{{\ell(\lambda_i)\leq r_i\atop 0\leq i\leq k}}}\mathcal{P}^{(0)}_{\lambda_0}\otimes \ldots\otimes \mathcal{P}^{(k)}_{\lambda_k},$$
where $\widehat{\bigoplus}$ denotes the orthogonal direct sum.

For given partition $\lambda_i$ of length $\leq r_i$, let
$K^{(i)}_{\lambda_i}(z_i;\overline{z_i})$ be the reproducing kernel
of $\mathcal{P}^{(i)}_{\lambda_i}$ with respect to \eqref{1.4}. By
Schur's lemma, there exists a positive constant
$c_{\lambda_0\ldots\lambda_k}$ such that
$c_{\lambda_0\ldots\lambda_k}\prod_{j=0}^kK^{(j)}_{\lambda_j}(z_j;\overline{z_j})$
is the reproducing kernel of $\mathcal{P}^{(0)}_{\lambda_0}\otimes
\ldots\otimes \mathcal{P}^{(k)}_{\lambda_k}$ with respect to the
above inner product $(\cdot,\cdot)$.

It is well known that the reproducing kernel can be expressed as a sum of square of the modules of the orthogonal basis, therefore we obtain
\begin{eqnarray*}
   & & \frac{\prod_{i=1}^k(\mu_i^{d_i}C_{\Omega_i})}{\pi^{n}}\int_{\left(\prod\limits_{j=1}^k\Omega_j\right)^{{\mathbb{B}}^{d_0}}(\mu)}c_{\lambda_0\ldots\lambda_k}\prod_{j=0}^kK^{(j)}_{\lambda_j}(z_j;\overline{z_j}) \prod_{i=1}^kN_{\Omega_i}(z_i,\overline{z_i})^{\mu_i((1+\nu_i)\alpha-d_0)-p_i}\\
   & &\times\Big(1-\|\widetilde{w}\|^2
\Big)^{\alpha-d_0-1}\prod_{i=1}^k\left(\nu_i+\frac{1}{1-\|\widetilde{w}\|^2}\right)^{d_i}\prod_{j=0}^kdm(z_j)  \\
   &=&\prod_{i=0}^k\dim\mathcal{P}^{(i)}_{\lambda_i}.
\end{eqnarray*}
It follows
\begin{equation}\label{e4.4}
  K_{\alpha}(Z;\overline{Z})=\sum_{{\ell(\lambda_i)\leq r_i\atop 0\leq
    i\leq k}}c_{\lambda_0\ldots\lambda_k}\prod_{j=0}^kK^{(j)}_{\lambda_j}(z_j;\overline{z_j})=\sum_{{\ell(\lambda_i)\leq r_i\atop 0\leq
    i\leq k}}\frac{\prod_{i=0}^k\dim\mathcal{P}^{(i)}_{\lambda_i}}{<\prod_{j=0}^kK^{(j)}_{\lambda_j}(z_j;\overline{z_j})>}\prod_{j=0}^kK^{(j)}_{\lambda_j}(z_j;\overline{z_j}),
\end{equation}
where $<f>$ denotes integral
\begin{eqnarray}
\nonumber <f>  &=& \frac{\prod_{i=1}^k(\mu_i^{d_i}C_{\Omega_i})}{\pi^{n}}\int_{\left(\prod_{j=1}^k\Omega_j\right)^{{\mathbb{B}}^{d_0}}(\mu)}f(Z)\times\prod_{i=1}^kN_{\Omega_i}(z_i,\overline{z_i})^{\mu_i((1+\nu_i)\alpha-d_0)-p_i} \\
\nonumber   & &
\times\Big(1-\|\widetilde{w}\|^2
\Big)^{\alpha-d_0-1}\prod_{i=1}^k\left(\nu_i+\frac{1}{1-\|\widetilde{w}\|^2}\right)^{d_i}\prod_{j=0}^kdm(z_j).
\end{eqnarray}

In order to compute the $K_{\alpha}(Z;\overline{Z})$, we just need to calculate the value of $ <\prod_{j=0}^kK^{(j)}_{\lambda_j}(z_j;\overline{z_j})>$. Since $\alpha>\max\{n,\frac{p_1-1}{\mu_1(1+\nu_1)},\ldots,\frac{p_k-1}{\mu_k(1+\nu_k)}\}$, then $\mu_i(1+\nu_i)\alpha-p_i>-1$ and $\alpha-n-1>-1$. Hence we can use (see \cite{F})
\begin{equation}\label{eq}
    \int_{\Omega}K_{\lambda}(z,\overline{z})N_{\Omega}(z,\overline{z})^sdm(z)=\frac{\dim \mathcal{P}_{\lambda}}{(p+s)_{\lambda}}
    \int_{\Omega}N_{\Omega}(z,\overline{z})^sdm(z)\;\;\;(s>-1)
\end{equation}
 to obtain
\begin{eqnarray*}
\nonumber   & & <\prod_{j=0}^kK^{(j)}_{\lambda_j}(z_j;\overline{z_j})> \\
\nonumber   &=&
\frac{\prod_{i=1}^k(\mu_i^{d_i}C_{\Omega_i})}{\pi^{n}}\prod_{i=1}^k\int_{\Omega_i}K^{(i)}_{\lambda_i}(z_i;\overline{z_i})N_{\Omega_i}(z_i,\overline{z_i})
^{\mu_i((1+\nu_i)\alpha+\lambda_0)-p_i}dm(z_i)\\
\nonumber   & &
\times \int_{{\mathbb{B}}^{d_0}} K^{(0)}_{\lambda_0}(w;\overline{w})(1-\|w\|^2)^{\alpha-d_0-1}\prod_{i=1}^k\left(\nu_i+\frac{1}{1-\|w\|^2}\right)^{d_i}dm(w)\\
\nonumber &=&
\frac{\prod_{i=1}^k(\mu_i^{d_i}C_{\Omega_i})}{\pi^{n}}\frac{\prod_{i=1}^k\dim\mathcal{P}^{(i)}_{\lambda_i}}{\prod_{i=1}^k(\mu_i((1+\nu_i)\alpha+\lambda_0))_{\lambda_i}^{(i)}}
\prod_{i=1}^k\int_{\Omega_i}N_{\Omega_i}(z_i,\overline{z_i})^{\mu_i((1+\nu_i)\alpha+\lambda_0)-p_i}
dm(z_i)\\
\nonumber   & &
\times \int_{{\mathbb{B}}^{d_0}} K^{(0)}_{\lambda_0}(w;\overline{w})(1-\|w\|^2)^{\alpha-d_0-1}\prod_{i=1}^k\left(\nu_i+\frac{1}{1-\|w\|^2}\right)^{d_i}dm(w).\\
\end{eqnarray*}
By \eqref{1.2}, we can simplify the above equation as follows
\begin{eqnarray}
\nonumber   & & <\prod_{j=0}^kK^{(j)}_{\lambda_j}(z_j;\overline{z_j})> \\
\nonumber &=&\frac{\prod_{i=1}^k(\mu_i^{d_i}\pi^{d_i})\cdot
}{\pi^{n}\prod_{i=1}^k\chi_i(\mu_i((1+\nu_i)\alpha+\lambda_0)-p_i)}
\frac{\prod_{i=1}^k\dim\mathcal{P}^{(i)}_{\lambda_i}}{\prod_{i=1}^k(\mu_i((1+\nu_i)\alpha+\lambda_0))_{\lambda_i}^{(i)}}\\
\label{e4.5}   & & \times
\underbrace{\int_{{\mathbb{B}}^{d_0}}K^{(0)}_{\lambda_0}(w;\overline{w})(1-\|w\|^2)^{\alpha-d_0-1}\prod_{i=1}^k\left(\nu_i+\frac{1}{1-\|w\|^2}\right)^{d_i}dm(w)}\limits_{\textcircled{1}}.
\end{eqnarray}
Since $\Omega_0=\mathbb{B}^{d_0}$, hence we have (\text{refer
to} \cite{Hua})
$$N_{{\mathbb{B}}^{d_0}}(w,\overline{w})=1-\Vert w\Vert^2,\;p_0=d_0+1,\;\chi_0(x)=(x+1)_{d_0},\; C_{\mathbb{B}^{d_0}}=1.$$
It follows
\begin{eqnarray*}
\nonumber \textcircled{1}  &=&\sum_{t_1=0}^{d_1}\cdots\sum_{t_k=0}^{d_k}
\left(\prod_{i=1}^k{d_i\choose t_i}\nu_i^{t_i}\right)
\int_{{\mathbb{B}}^{d_0}}
K^{(0)}_{\lambda_0}(w;\overline{w})N_{{\mathbb{B}}^{d_0}}(w,\overline{w})^{\alpha-n-1+\sum_{j=1}^kt_j}dm(w) \\
\nonumber &=&\sum_{t=0}^d\left(\sum_{\sum_{i=1}^kt_i=t\atop
t_i\geq 0,1\leq i\leq k}\prod_{i=1}^k{d_i\choose
t_i}\nu_i^{t_i}\right)\frac{\dim
\mathcal{P}^{(0)}_{\lambda_0}}{(\alpha-d+t)_{\lambda_0}}\int_{{\mathbb{B}}^{d_0}}
N_{{\mathbb{B}}^{d_0}}(w,\overline{w})^{\alpha-n-1+t}dm(w)\\
\end{eqnarray*}
by \eqref{eq}. Applying the formula \eqref{1.2} again, we know that
\begin{eqnarray*}
\nonumber \textcircled{1} &=&\sum_{t=0}^d\left(\sum_{\sum_{i=1}^kt_i=t\atop
t_i\geq 0,1\leq i\leq k}\prod_{i=1}^k{d_i\choose
t_i}\nu_i^{t_i}\right)\frac{\dim
\mathcal{P}^{(0)}_{\lambda_0}}{(\alpha-d+t)_{\lambda_0}}\frac{\pi^{d_0}}{\chi_0(\alpha-n-1+t)}.\\
&=&\sum_{t=0}^d\left(\sum_{\sum_{i=1}^kt_i=t\atop
t_i\geq 0,1\leq i\leq k}\prod_{i=1}^k{d_i\choose
t_i}\nu_i^{t_i}\right)\frac{\dim
\mathcal{P}^{(0)}_{\lambda_0}}{(\alpha-d+t)_{\lambda_0}}\frac{\pi^{d_0}}{(\alpha-n+t)_{d_0}}.
\end{eqnarray*}

By the formula \eqref{s}, we have
\begin{eqnarray*}
\nonumber  \textcircled{1} &=&\dim
\mathcal{P}_{\lambda_0}^{(0)}\pi^{d_0}\sum_{t=0}^d\left(\sum_{\sum_{i=1}^kt_i=t\atop
t_i\geq 0,1\leq i\leq k}\prod_{i=1}^k{d_i\choose
t_i}\nu_i^{t_i}\right)\frac{\Gamma(\alpha-d+t)}{\Gamma(\alpha-d+t+\lambda_0)}\frac{\Gamma(\alpha-n+t)}{\Gamma(\alpha-n+t+d_0)}
\end{eqnarray*}
Since $n=d+d_0$, it follows
\begin{eqnarray*}
\nonumber\textcircled{1} &=&
\dim
\mathcal{P}_{\lambda_0}^{(0)}\pi^{d_0}\sum_{t=0}^d\left(\sum_{\sum_{i=1}^kt_i=t\atop
t_i\geq 0,1\leq i\leq k}\prod_{i=1}^k{d_i\choose
t_i}\nu_i^{t_i}\right)\frac{\Gamma(\alpha+\lambda_0)}{\Gamma(\alpha-d+t+\lambda_0)}\frac{\Gamma(\alpha-n+t)}{\Gamma(\alpha-n)}\frac{\Gamma(\alpha-n)}{\Gamma(\alpha+\lambda_0)}\\
\nonumber &=&\dim
\mathcal{P}_{\lambda_0}^{(0)}\pi^{d_0}\sum_{t=0}^d\left(\sum_{\sum_{i=1}^kt_i=t\atop
t_i\geq 0,1\leq i\leq k}\prod_{i=1}^k{d_i\choose
t_i}\nu_i^{t_i}\right)\frac{\big(\alpha+\lambda_0-(d-t)\big)_{d-t}\big(\alpha-n\big)_t}{\big(\alpha-n\big)_{n+\lambda_0}}\\
&=&\frac{\pi^{d_0}\dim
\mathcal{P}_{\lambda_0}^{(0)}}{\big(\alpha-n\big)_{n+\lambda_0}}\sum_{t=0}^d\sigma(t)\big(\alpha+\lambda_0-t\big)_{t}\big(\alpha-n\big)_{d-t}.
\end{eqnarray*}
Therefore, the formulas \eqref{e4.5}, \eqref{e4.5.2} and a direct computation imply
\begin{eqnarray}
\nonumber & &    <\prod_{j=0}^kK^{(j)}_{\lambda_j}(z_j;\overline{z_j})>\\
\nonumber &=&\frac{\prod_{i=1}^k\mu_i^{d_i}\prod_{i=0}^k\dim\mathcal{P}^{(i)}_{\lambda_i}\sum_{t=0}^d\sigma(t)\big(\alpha+\lambda_0-t\big)_{t}
\big(\alpha-n\big)_{d-t}}{\big(\alpha-n\big)_{n+\lambda_0}\prod_{i=1}^k\chi_i(\mu_i((1+\nu_i)\alpha+\lambda_0)-p_i)\prod_{i=1}^k(\mu_i((1+\nu_i)\alpha+\lambda_0))_{\lambda_i}^{(i)}}\\
 \nonumber &=&\frac{\prod_{i=0}^k\dim\mathcal{P}^{(i)}_{\lambda_i}}{\psi(\alpha,\lambda_0)(\alpha-n)_{n+\lambda_0}\prod_{i=1}^k(\mu_i((1+\nu_i)\alpha+\lambda_0))_{\lambda_i}^{(i)}}.
\end{eqnarray}

Hence, combining  \eqref{e4.4} and \eqref{1.7}, we have
\begin{eqnarray*}
\nonumber   & & K_{\alpha}(Z;\overline{Z}) \\
\nonumber    &=&
  \sum_{{\ell(\lambda_i)\leq r_i\atop 0\leq i\leq k}}
  \psi(\alpha,\lambda_0)(\alpha-n)_{n+\lambda_0}K^{(0)}_{\lambda_0}(w;\overline{w})\prod_{i=1}^k(\mu_i((1+\nu_i)\alpha+\lambda_0))_{\lambda_i}^{(i)}K^{(j)}_{\lambda_j}(z_j;\overline{z_j})
      \\
\nonumber   &=&
\sum_{{\ell(\lambda_0)\leq r_0}}\prod_{i=1}^k \frac{1}{N_{\Omega_i}(z_i,\overline{z_i})^{\mu_i((1+\nu_i)\alpha+\lambda_0)}} \psi(\alpha,\lambda_0)(\alpha-n)_{n+\lambda_0}K^{(0)}_{\lambda_0}(w;\overline{w})   \\
\end{eqnarray*}
Since $r_0=1$, $\lambda_0\in \mathbb{N}$ and $K^{(0)}_{\lambda_0}(w;\overline{w})$ be the reproducing kernel
of $\mathcal{P}^{(0)}_{\lambda_0}$ with respect to \eqref{1.4} where $\mathcal{P}^{(0)}_{\lambda_0}$ is the space of homogeneous holomorphic
polynomials of degree $\lambda_0$, Hence we have
$$K^{(0)}_{\lambda_0}(w;\overline{w})=\sum\limits_{\vert \alpha\vert=\lambda_0}\frac{\vert w\vert^{2\alpha}}{\prod\limits_{i=1}^{d_0}\Gamma(\alpha_i+1)}=\frac{\|w\|^{2\lambda_0}}{\lambda_0!}.$$
Consequently, we obtain
\begin{eqnarray}
\nonumber & &  K_{\alpha}(Z;\overline{Z})\\
\nonumber   &=&
\sum_{{\lambda_0=0}}^{+\infty}\prod_{i=1}^k \frac{1}{N_{\Omega_i}(z_i,\overline{z_i})^{\mu_i((1+\nu_i)\alpha+\lambda_0)}} \psi(\alpha,\lambda_0)(\alpha-n)_{n+\lambda_0}\frac{\|w\|^{2\lambda_0}}{\lambda_0!} \\
\nonumber &=&\frac{1}{{\prod_{i=1}^kN_{\Omega_i}(z_i,\overline{z_i})^{\mu_i(1+\nu_i)\alpha}}}
\sum_{t=0}^{+\infty}\psi(\alpha,t)\frac{(\alpha-n)_{n+t}}{t!}\|\widetilde{w}\|^{2t}\\
\nonumber &=&\frac{(\alpha-n)_n}{{\prod_{i=1}^kN_{\Omega_i}(z_i,\overline{z_i})^{\mu_i(1+\nu_i)\alpha}}}
\sum_{t=0}^{+\infty}\psi(\alpha,t)\frac{(\alpha)_t}{t!}\|\widetilde{w}\|^{2t}
\end{eqnarray}
by the following fact (see \eqref{s})
$$(\alpha-n)_{n+t}=(\alpha-n)_n(\alpha)_t.$$
 The proof is finished.

\end{proof}

Now, combining Theorem \ref{Th:2.3}, we can get the expression of  the Rawnsley's $\varepsilon$-function
for the generalized Cartan-Hartogs domain
$\big(\prod_{j=1}^k\Omega_j\big)^{{\mathbb{B}}^{d_0}}(\mu)$  with the metric $g(\mu;\nu)$.

\begin{proof}[Proof of Theorem \ref{Th:2.4}]
Using the definition \eqref{eq1.4}, we have
\begin{equation*}
   \varepsilon_{(\alpha,g(\mu;\nu))}(z_1,\ldots,z_k,w)= e^{-\alpha\Phi(z_1,\ldots,z_k,w)} K_{\alpha}(z_1,\ldots,z_k,w;\overline{z_1},\ldots,\overline{z_k},\overline{w}),
\end{equation*}
and by the definition \eqref{eq1.2}, we get
$$e^{-\alpha\Phi(z_1,\ldots,z_k,w)}={\prod_{i=1}^kN_{\Omega_i}(z_i,\overline{z_i})^{\mu_i(1+\nu_i)\alpha}}(1-\|\widetilde{w}\|^2)^{\alpha}.$$
Therefore, by \eqref{eq2.26}, we obtain \eqref{eq2.37}. The proof is finished.
\end{proof}

As a consequence of Theorem \ref{Th:2.4}, we give the proof of
Corollary \ref{Th:3.1}.

\begin{proof}[Proof of Corollary \ref{Th:3.1}]
From \eqref{eq2.37} and \eqref{eq2.38}, we obtain
\begin{equation}\label{eq2.42}
    \sum_{j=0}^dc_j(\alpha)(1-\|w\|^2)^{-(\alpha-(d-j))}=\sum_{t=0}^{+\infty}\psi(\alpha,t)\frac{(\alpha-n)_{n+t}}{t!}\|w\|^{2t}.
\end{equation}
Using
\begin{equation}\label{eq2.43}
    (1-\|w\|^2)^{-\alpha}=\sum_{t=0}^{+\infty}\frac{(\alpha)_{t}}{t!}\|w\|^{2t},
\end{equation}
we get
\begin{equation}\label{eq2.44}
    \sum_{j=0}^d\frac{c_j(\alpha)}{(\alpha-n)_{d_0+j}}(\alpha+t-d)_j=\psi(\alpha,t)(\alpha+t-d)_d.
\end{equation}
This indicates that $\psi(\alpha,t)(\alpha+t-d)_d$ is a polynomial
of $\alpha+t$. It follow $\phi(\alpha+t)=\psi(\alpha,t)(\alpha+t-d)_d$ is a polynomial of $\alpha+t$.

It is known that if $\phi(x)$ is a polynomial of $x$, then we have (refer to \cite{F})
\begin{equation*}
    \phi(x)=\sum_{j=0}^d\frac{D^j\phi(d)}{j!}(x-d)_j.
\end{equation*}
Hence the formula \eqref{eq2.44} can be written as
\begin{equation*}
    \sum_{j=0}^d\frac{c_j(\alpha)}{(\alpha-n)_{d_0+j}}(\alpha+t-d)_j=\sum_{j=0}^d\frac{D^j\phi(d)}{j!}(\alpha+t-d)_j.
\end{equation*}
Therefore, we obtain
\begin{equation}\label{eq2.45}
   c_j(\alpha)= \frac{D^j\phi(d)}{j!}(\alpha-n)_{d_0+j}.
\end{equation}
 Substituting \eqref{eq2.45} into \eqref{eq2.38}, we obtain \eqref{eq2.40}.
\end{proof}

\setcounter{equation}{0}
\section{The proof of Theorems 1.4 and 1.5}

In order to prove that a quantization procedure can be established on the generalized Cartan-Hartogs domain $\big(\big(\prod_{j=1}^k\Omega_j\big)^{{\mathbb{B}}^{d_0}}(\mu), g({\mu};\nu)\big)$. we just need to prove that the Calabi's diastasis function
$D_{ g(\mu;\nu)}$ and the Rawnsley's $\varepsilon$-function associated to the metric $ g(\mu;\nu)$ satisfies the condition $(\uppercase\expandafter{\romannumeral1})$ and $(\uppercase\expandafter{\romannumeral2})$, respectively. Before we do the proof, we will give the following results.

\begin{Lemma}\label{lemma,inequlity}
Let $\big(\prod_{j=1}^{k}\Omega_{j}\big)^{\mathbb{B}^{d_{0}}}(\mu)$ be the generalized Cartan-Hartogs domain, $\mu=(\mu_1,\ldots,\mu_k)$ and $\mu_i\in W(\Omega_{i})\backslash\{0\}$, here $W(\Omega_{i})$ are the Wallach sets of $\Omega_i$, $1\leq i\leq k$. Then we have
\begin{equation}
\sup\limits\limits_{(z,w),(\xi,\eta)\in\big(\prod\limits_{j=1}^{k}\Omega_{j}\big)^{\mathbb{B}^{d_{0}}}(\mu)}\left|1-
w{\overline{\eta}}^{\mathrm{t}}\prod
\limits_{i=1}^{k}N_{\Omega_{i}}(z_{i},\overline{\xi_{i}})^{-\mu_{i}}\right|<+\infty,
\end{equation}
where $z=(z_{1},\ldots,z_{k})$ and $\xi=(\xi_{1},\ldots,\xi_{k})$.
\end{Lemma}
\begin{proof}[Proof]
Since $\mu_i\in W(\Omega_{i})\backslash\{0\}$, we know that $\prod
_{i=1}^{k}N_{\Omega_{i}}(z_{i},
\overline{z_{i}})^{-\mu_{i}}$ is the reproducing kernel of some Hilbert space (see \cite{FK}). Hence there exists an orthonormal basis $\{g_{j}\}$ such that
$$\prod\limits_{i=1}^{k}N_{\Omega_{i}}(z_{i},
\overline{\xi_{i}})^{-\mu_{i}}=\sum\limits_{j}g_{j}(z)\overline{g_{j}(\xi)},$$
where $z=(z_{1},\ldots,z_{k})$ and $\xi=(\xi_{1},\ldots,\xi_{k})$. Then Cauchy-Schwarz inequality yields 
\begin{equation*}
\prod\limits_{i=1}^{k}\left|N_{\Omega_{i}}(z_{i},\overline{\xi_{i}})^{-\mu_{i}}\right|\leq \prod\limits_{i=1}^{k}N_{\Omega_{i}}(z_{i},
\overline{z_{i}})^{\frac{-\mu_{i}}{2}}\prod\limits_{i=1}^{k}N_{\Omega_{i}}(\xi_{i},
\overline{\xi_{i}})^{\frac{-\mu_{i}}{2}}.
\end{equation*}
Therefore, by $(z,w), (\xi,\eta)\in \big(\prod_{j=1}^{k}\Omega_{j}\big)^{\mathbb{B}^{d_{0}}}(\mu)$, we get the following inequality:
\begin{equation}\label{inequality}
\vert w{\overline{\eta}}^{\mathrm{t}}\vert\prod\limits_{i=1}^{k}\left|N_{\Omega_{i}}(z_{i},
\overline{\xi_{i}})^{-\mu_{i}}\right|\leq \Vert w\Vert\Vert \eta\Vert\prod\limits_{i=1}^{k}N_{\Omega_{i}}(z_{i},
\overline{z_{i}})^{\frac{-\mu_{i}}{2}}\prod\limits_{i=1}^{k}N_{\Omega_{i}}(\xi_{i},
\overline{\xi_{i}})^{\frac{-\mu_{i}}{2}}<1.
\end{equation}
The proof is finished.
\end{proof}

\begin{Lemma}\label{Thm.5.1}
 The noncompact K\"{a}hler manifold $\big(\big(\prod_{j=1}^{k}\Omega_{j})^{\mathbb{B}^{d_{0}}}(\mu), g({\mu};\nu)\big)$ satisfies condition $(\uppercase\expandafter{\romannumeral1})$.
\end{Lemma}

\begin{proof}[Proof]
To prove that the noncompact K\"{a}hler manifold $\big(\big(\prod_{j=1}^{k}\Omega_{j})^{\mathbb{B}^{d_{0}}}(\mu), g({\mu};\nu)\big)$ satisfies condition $(\uppercase\expandafter{\romannumeral1})$, we only need to prove that the noncompact K\"{a}hler manifold $\big(\big(\prod_{j=1}^{k}\Omega_{j})^{\mathbb{B}^{d_{0}}}(\mu),\beta g({\mu};\nu)\big)$ satisfies condition $(\uppercase\expandafter{\romannumeral1})$ for given $\beta>\max\left\{\frac{(r_{1}-1)a_{1}}{2\mu_{1}(1+\nu_1)},\ldots,\frac{(r_{k}-1)a_{k}}{2\mu_{k}(1+\nu_k)}\right\}$.

By definition of the Calabi's diastasis function of $\big(\big(\prod_{j=1}^{k}\Omega_{j}\big)^{\mathbb{B}^{d_{0}}}(\mu),\beta g({\mu};\nu)\big)$,
 we have

\begin{eqnarray}
\nonumber   & & \exp\{-D_{\beta g(\mu;\nu)}((z,w),(\xi,\eta))\} \\
\nonumber   &=& {\left|\prod\limits_{i=1}^{k} N_{\Omega_{i}}(z_{i},\overline{\xi_{i}})^{-\beta \mu_{i}(1+\nu_{i})} \left(1
-\frac{w\overline{\eta}^{t}}{\prod\limits_{i=1}^{k}N_{\Omega_{i}}(z_{i},\overline{\xi_{i}})^{\mu_{i}}}\right)^{-\beta}\right|}^{2} \\
\nonumber   & &\times \prod\limits_{i=1}^{k}N_{\Omega_{i}}(z_{i},\overline{z_{i}})^{\beta \mu_{i}(1+\nu_{i})}
\prod\limits_{i=1}^{k}N_{\Omega_{i}}(\xi_{i},\overline{\xi_{i}})^{\beta \mu_{i}(1+\nu_{i})} \\
\label{eq6.4}   & &\times \left(1-\frac{{\Vert w\Vert}^{2}}{\prod\limits_{i=1}^{k}N_{\Omega_{i}}(z_{i},\overline{z_{i}})^{\mu_{i}}}\right)^{\beta}
\left(1-\frac{{\Vert \eta\Vert}^{2}}{\prod\limits_{i=1}^{k}N_{\Omega_{i}}(\xi_{i},\overline{\xi_{i}})^{\mu_{i}}}\right)^{\beta}.
\end{eqnarray}

By (\ref{1.7}), we have
\begin{eqnarray}
\nonumber   & &\prod\limits_{i=1}^{k} N_{\Omega_{i}}(z_{i},\overline{\xi_{i}})^{-\beta \mu_{i}(1+\nu_{i})}
 \left(1-\frac{w\overline{\eta}^{t}}{\prod\limits_{i=1}^{k}N_{\Omega_{i}}(z_{i},\overline{\xi_{i}})^{\mu_{i}}}\right)^{-\beta}  \\
\nonumber   &=& \sum_{j=0}^{+\infty}\frac{\Gamma(\beta+j)}{\Gamma(\beta)j!}
\frac{(w\overline{\eta}^{t})^j}{\prod\limits_{i=1}^{k}N_{\Omega_{i}}(z_{i},\overline{\xi_{i}})^{\mu_{i}(\beta(1+\nu_i)+j)}} \\
\label{eq6.5}   &=& \sum_{j=0}^{+\infty}\sum_{\ell(\lambda^{(i)})\leq r_i\atop 1\leq i\leq k}\left(\frac{\Gamma(\beta+j)}{\Gamma(\beta)}\prod_{i=1}^k(\beta\mu_i(1+\nu_i)+\mu_i j)^{(i)}_{\lambda^{(i)}}\right)
  \left( \frac{(w\overline{\eta}^{t})^j}{j!}\prod_{i=1}^kK^{(i)}_{\lambda^{(i)}}(z_i,\overline{\xi_i})\right),
\end{eqnarray}
where
$$\frac{1}{N_{\Omega_{i}}(z_{i},\overline{\xi_{i}})^s}=\sum_{\ell(\lambda^{(i)})\leq r_i}(s)^{(i)}_{\lambda^{(i)}}K^{(i)}_{\lambda^{(i)}}(z_i,\overline{\xi_i})$$
and
$$ (s)^{(i)}_{\lambda^{(i)}}=\prod_{j=1}^{r_i}\left(s-\frac{j-1}{2}a_i\right)_{\lambda^{(i)}_j}.$$

When $\beta\mu_i(1+\nu_i)>\frac{r_i-1}{2}a_i$, $1\leq i\leq k$,  namely $\beta>\max\left\{\frac{(r_{1}-1)a_{1}}{2\mu_{1}(1+\nu_1)},\ldots,\frac{(r_{k}-1)a_{k}}{2\mu_{k}(1+\nu_k)}\right\}$, we have
$$(\beta\mu_i(1+\nu_i)+\mu_i j)^{(i)}_{\lambda^{(i)}}>0$$
for  $\forall\;j\geq 0$ and $\forall\; \lambda^{(i)}$ with $\ell(\lambda^{(i)})\leq r_i$.

Since $K^{(i)}_{\lambda^{(i)}}(z_i,\overline{\xi_i})$ and $\frac{(w\overline{\eta}^{t})^j}{j!}$ are  reproducing kernels, then
$$\left(\frac{\Gamma(\beta+j)}{\Gamma(\beta)}\prod_{i=1}^k(\beta\mu_i(1+\nu_i)+\mu_i j)^{(i)}_{\lambda^{(i)}}\right)
  \left( \frac{(w\overline{\eta}^{t})^j}{j!}\prod_{i=1}^kK^{(i)}_{\lambda^{(i)}}(z_i,\overline{\xi_i})\right)$$
are also reproducing kernels. Observe that $K^{(i)}_{\lambda^{(i)}}(z_i,\overline{\xi_i})$ is the reproducing kernels of $\mathcal{P}^{(i)}_{\lambda_i}$ where $\mathcal{P}^{(i)}_{\lambda_i}$ is the subspace of homogeneous holomorphic polynomials of degree $\vert \lambda_i\vert$, therefore, there are linearly independent holomorphic homogeneous polynomials $f_l$ ($1\leq l<+\infty$) such that
\begin{equation}\label{eq6.6}
\prod\limits_{i=1}^{k} N_{\Omega_{i}}(z_{i},\overline{\xi_{i}})^{-\beta \mu_{i}(1+\nu_{i})}
 \left(1-\frac{w\overline{\eta}^{t}}{\prod\limits_{i=1}^{k}N_{\Omega_{i}}(z_{i},\overline{\xi_{i}})^{\mu_{i}}}\right)^{-\beta} =\sum_{l=1}^{+\infty}f_l(z,w)\overline{f_{l}(\xi,\eta)}.
\end{equation}

Applying the Cauchy-Schwarz inequality to \eqref{eq6.4}, we have
$$\exp\{-D_{\beta g(\mu;\nu)}((z,w),(\xi,\eta))\}\leq1.$$
By $f(z,w)=(f_1(z,w),f_2(z,w),\ldots)\neq 0$, we get $\exp\{-D_{\beta g(\mu;\nu)}((z,w),(\xi,\eta))\}=1$ if and only if there exists a constant $c$ such that
\begin{equation}\label{eq6.7}
f_{l}(z,w)=cf_{l}(\xi,\eta),\quad \forall l\in \mathbb{N}^{+}.
\end{equation}

Taking $f_l=1$ in \eqref{eq6.7}, we obtain $c=1$.

Taking linearly independent holomorphic homogeneous polynomials $f_l$ of degree 1  in  \eqref{eq6.7}, it follows $(z,w)=(\xi,\eta)$.

Conversely, if $(z,w)=(\xi,\eta)$, it is not hard to see that
$$\exp\{-D_{\beta g(\mu;\nu)}((z,w),(\xi,\eta))\}=1.$$
So far, we complete the proof.
\end{proof}

Now we can give the proof of the Theorem  \ref{Th:1.7}.
\begin{proof}[The Proof of Theorem \ref{Th:1.7}]

By Lemma \ref{Thm.5.1}, we know that condition $(\uppercase\expandafter{\romannumeral1})$ is satisfied.

In the following, we prove the condition $(\uppercase\expandafter{\romannumeral2})$ can be also fulfilled. Firstly, by the explicit expression of $\Phi$, the function $a((z,w),\overline{(z,w)})$ admits a sesquianalytic extension on $\big(\prod\limits_{j=1}^{k}\Omega_{j}\big)^{\mathbb{B}^{d_{0}}}(\mu)\times \big(\prod\limits_{j=1}^{k}\Omega_{j}\big)^{\mathbb{B}^{d_{0}}}(\mu)$.

 We define an infinite set $E$ as follows
$$E:=\left\{\alpha\in\mathbb{{N}}:\alpha>\max\left\{n,\frac{p_1-1}{\mu_1(1+\nu_1)},\ldots,\frac{p_k-1}{\mu_k(1+\nu_k)}\right\}\right\}.$$

Let $\alpha\in E$,  by \eqref{eq2.40}, we know that
\begin{equation*}
\varepsilon_{(\alpha, g(\mu;\nu))}(z,w;\overline{\xi},\overline{\eta})=\sum_{j=0}^d\frac{D^j\phi(d)}{j!}(\alpha-n)_{d_0+j}X(z,w;\overline{\xi},\overline{\eta})^{d-j},
\end{equation*}
where
$$X(z,w;\overline{\xi},\overline{\eta})=1-\frac{{ w\overline{\eta}^{\mathrm{t}}}}{\prod\limits_{i=1}^{k}N_{\Omega_{i}}(z_{i},
\overline{\xi_{i}})^{\mu_{i}}}.$$

Since $\frac{D^{d}\phi(d)}{d!}=1$,  we have
\begin{eqnarray*}
 \varepsilon_{(\alpha, g(\mu;\nu))}(z,w;\overline{\xi},\overline{\eta})  &=& \alpha^n+ \sum_{j=1}^na_j(z,w;\overline{\xi},\overline{\eta})\alpha^{n-j} \\
   &=& \alpha^{n}+B(z,w;\overline{\xi},\overline{\eta})\alpha^{n-1}+C(\alpha,z,w;\overline{\xi},\overline{\eta})\alpha^{n-2},
\end{eqnarray*}
where

\begin{equation}\label{B(z,w)}
  B(z,w;\overline{\xi},\overline{\eta})=-\frac{(n+1)n}{2}+\frac{D^{d-1}\phi(d)}
{(d-1)!}X(z,w;\overline{\xi},\overline{\eta}),
\end{equation}

\begin{equation}\label{C(z,w)}
C(\alpha,z,w;\overline{\xi},\overline{\eta})
=\sum_{j=2}^na_j(z,w;\overline{\xi},\overline{\eta})\alpha^{2-j},
\end{equation}
here $a_j(z,w;\overline{\xi},\overline{\eta})$ are polynomials in $X(z,w;\overline{\xi},\overline{\eta})$.

By Lemma \ref{lemma,inequlity}, it easy to see that
\begin{equation*}
\sup\limits\limits_{(z,w),(\xi,\eta)\in\big(\prod\limits_{j=1}^{k}\Omega_{j}\big)^{\mathbb{B}^{d_{0}}}(\mu)}\vert a_j(z,w;\overline{\xi},\overline{\eta}) \vert<+\infty,\;2\leq j\leq n.
\end{equation*}
So, from  \eqref{B(z,w)} and \eqref{C(z,w)}, we get
\begin{equation*}
\sup\limits\limits_{(z,w),(\xi,\eta)\in\big(\prod\limits_{j=1}^{k}\Omega_{j}\big)^{\mathbb{B}^{d_{0}}}(\mu)}\vert B(z,w;\overline{\xi},\overline{\eta}) \vert<+\infty
\end{equation*}
and
\begin{equation*}
\sup\limits\limits_{(z,w),(\xi,\eta)\in\big(\prod\limits_{j=1}^{k}\Omega_{j}\big)^{\mathbb{B}^{d_{0}}}(\mu),\;\alpha\in E}\vert C(\alpha,z,w;\overline{\xi},\overline{\eta})\vert<+\infty.
\end{equation*}
Hence the condition $(\uppercase\expandafter{\romannumeral2})$ is verified. So far we complete the proof.
\end{proof}
Lastly, we give the proof of Theorem \ref{co1.5}.
\begin{proof}[The Proof of Theorem \ref{co1.5}]
By definition, the metric $\alpha g(\mu,\nu)$ is balanced if and only if $\varepsilon_{(\alpha,g(\mu;\nu))}(z,w)$ is dependent of $(z,w)$. The formula \eqref{eq2.37} implies that $\alpha g(\mu,\nu)$ is balanced iff there exist a constant $\lambda(\alpha)$ with respect to $(z,w)$ such that
$$(1-\|\widetilde{w}\|^2)^{-\alpha}
=\lambda(\alpha)\sum_{t=0}^{+\infty}\psi(\alpha,t)\frac{(\alpha)_t}{t!}\|\widetilde{w}\|^{2t}.$$
Thus, by \eqref{eq2.43}, we conclude that $\lambda(\alpha)\psi(\alpha,t)=1$, which means that $\psi(\alpha,t)$ is a constant with respect to $t$. By the expression of $\psi(\alpha,t)$, that is,
$$\psi(\alpha,t)=\frac{\prod_{i=1}^k\chi_i(\mu_i((1+\nu_i)\alpha+t)-p_i)}{\prod_{i=1}^k\mu_i^{d_i}\sum_{j=0}^d\sigma(j)(\alpha-n)_{d-j}(\alpha+t-j)_j},$$
we have $\psi(\alpha,t)$ tends to $1$ as $t\rightarrow\infty$. Hence the metrics $\alpha g(\mu,\nu)$ is balanced if and only if
$$\prod\limits_{i=1}^{k}\chi_{i}(\mu_{i}((1+\nu_{i})x+y)-p_{i})=\left(\prod\limits_{i=1}^{k}\mu_{i}^{d_{i}}\right)\sum\limits_{t=0}^{d}\left(\sum_{\sum_{i=1}^kt_i=t\atop
t_i\geq 0,1\leq i\leq k}\prod_{i=1}^k{d_i\choose t_i}\nu_i^{t_i}\right)(x-n)_{t}(x+y-d+t)_{d-t}$$
by \eqref{f1}.

 Now we turn into the rest proof of Theorem\ref{co1.5}. On one hand, by Lemma \ref{Thm.5.1}, we know that Condition $(\uppercase\expandafter{\romannumeral1})$ is satisfied. On the other hand, the formula \eqref{Blanced} implies 
 $$\varepsilon_{(\alpha, g(\mu;\nu))}(z,w;\overline{\xi},\overline{\eta})=(\alpha-n)_n=(\alpha-1)(\alpha-2)\cdots(\alpha-n).$$
 Therefore, by Berezin \cite{Berezin}, we have $\big(\big(\prod_{j=1}^k\Omega_j\big)^{{\mathbb{B}}^{d_0}}(\mu), g(\mu;\nu)\big)$ admits a Berezin quantization for $\mu_{i}\in W(\Omega_{i})\backslash\{0\}$ $(1\leq i\leq k)$. The proof is completed.
\end{proof}
\vskip 10pt

 \noindent\textbf{Acknowledgments}\quad We sincerely thank the referees, who read the paper very carefully and gave many useful suggestions. E. Bi was supported by the Natural Science Foundation of Shandong Province, China (No. ZR2018BA015), Z. Feng was supported by the Scientific Research Fund of Sichuan Provincial Education Department (No. 18ZB0272) and Z. Tu was supported by
the National Natural Science Foundation of China (No. 11671306).


\addcontentsline{toc}{section}{References}
\phantomsection
\renewcommand\refname{References}
\small{
}
\clearpage
\end{document}